\documentclass[12pt]{amsart}
\usepackage{latexsym,fancyhdr,amssymb,color,amsmath,amsthm,graphicx,listings,comment}
\usepackage[section]{placeins}
\newtheorem{thm}{Theorem} 
\newtheorem{lemma}{Lemma} 
\newtheorem{coro}{Corollary}
\let\paragraph\subsection

\title{Soft Barycentric Refinements}
\author{Oliver Knill}
\date{March 2, 2025}
\address{Department of Mathematics \\ Harvard University \\ Cambridge, MA, 02138 }


\begin{document}
\maketitle

\begin{abstract}
The soft Barycentric refinement preserves manifolds with or without boundary.
In every dimension larger than one, there is a universal spectral 
central limiting measure that has affinities with the Barycentric 
limiting measure one dimension lower. Ricci type quantities like the length of the 
dual sphere of co-dimension-2 simplex stay invariant under soft 
refinements. We prove that the dual graphs of any manifold can be colored with 
3 colors, which is in the 2-dimensional case a special case of the Gr\"otzsch theorem. 
It follows that the vertices of a soft Barycentric refined $q$-manifold $G'$ 
can be colored by $q+1$ or $q+2$ colors. 
\end{abstract}

\section{Introduction}

\paragraph{}
A {\bf finite simple graph} $G=(V,E)$ defines a {\bf finite abstract simplicial 
complex} $V_1=\{ x=V(H), H \sim K_k \; H \; {\rm subgraph} \; of G\}$, a finite
set of non-empty sets closed under the operation of taking finite non-empty subsets.
This set of sets $V_1$ defines the {\bf Barycentric refinement graph}
$\psi(G)=G_1=(V_1,E_1)$ with vertex set $V_1$ and edge set 
$E_1=\{ (a,b) \in V_1 \times V_1, a \neq b, a \subset b | b \subset a \}$.
We define here a {\bf soft Barycentric refinement}. In manifolds without boundary
it disregards the co-dimension one simplices $W$ as vertices and adds edges between
maximal simplices which have an intersection in $W$. It has similar properties than
the Barycentric refinement but it does not increase vertex degrees in 2 dimensions.
We focus in this article on two subjects, 
{\bf 1)} the soft Barycentric limit and 
{\bf 2)} the chromatic number of softly refined manifolds. 
In both of these two stories, the dual $\hat{G}$ of a manifold $G$ plays 
an important role. The skeleton graph $\hat{G}$ has the facets of $G$ 
as points and connects two if they intersect in $W$. 
If the manifold $G$ is orientiable, it
can naturally be given a {\bf cell structure} making
a dual manifold of $G$ with reversed f-vector $(f_d, \dots, f_0)$ of $G$ 
and reversed Betti vector $(b_d, \dots, b_0)$ of $G$.

\begin{figure}[!htpb]
\scalebox{0.15}{\includegraphics{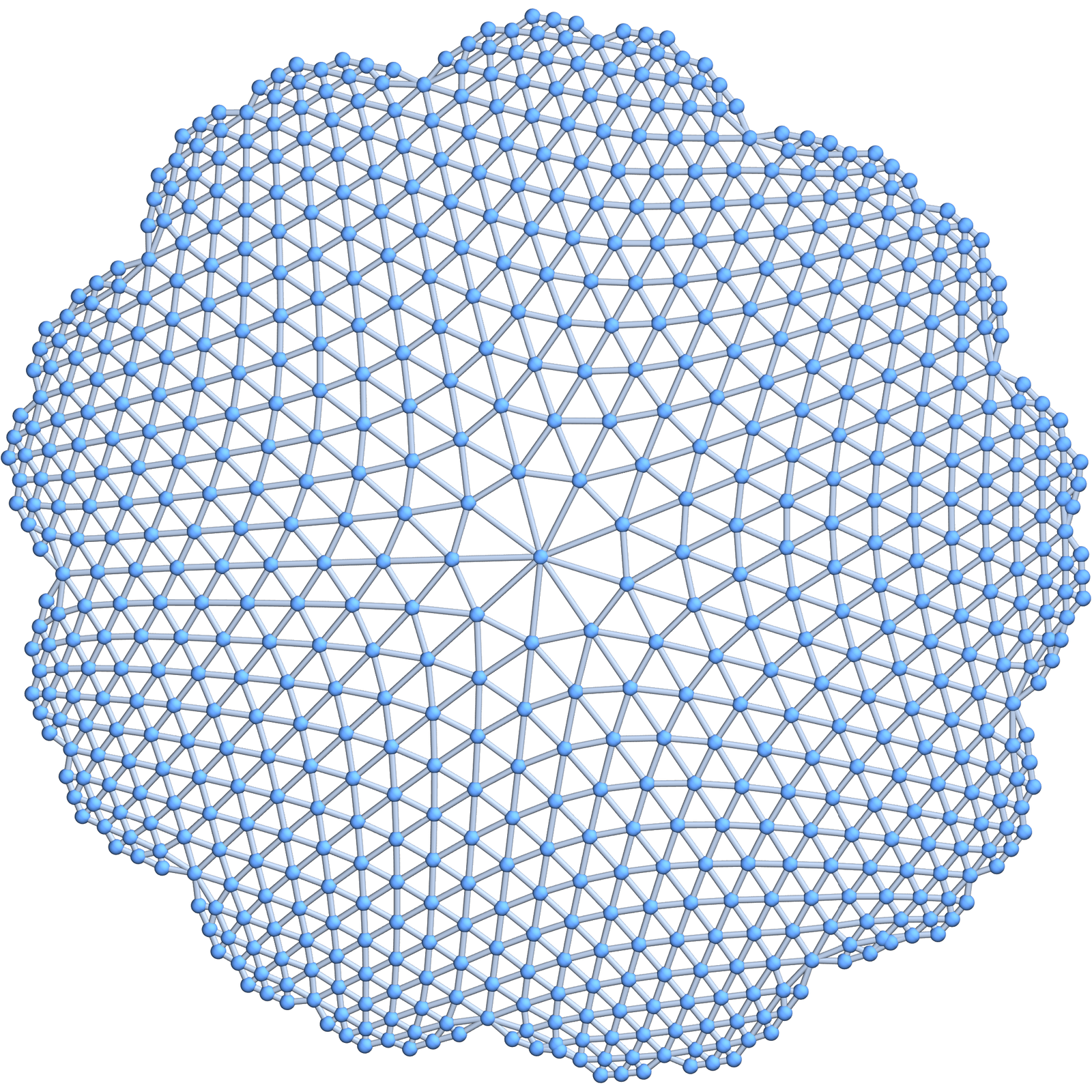}}
\scalebox{0.15}{\includegraphics{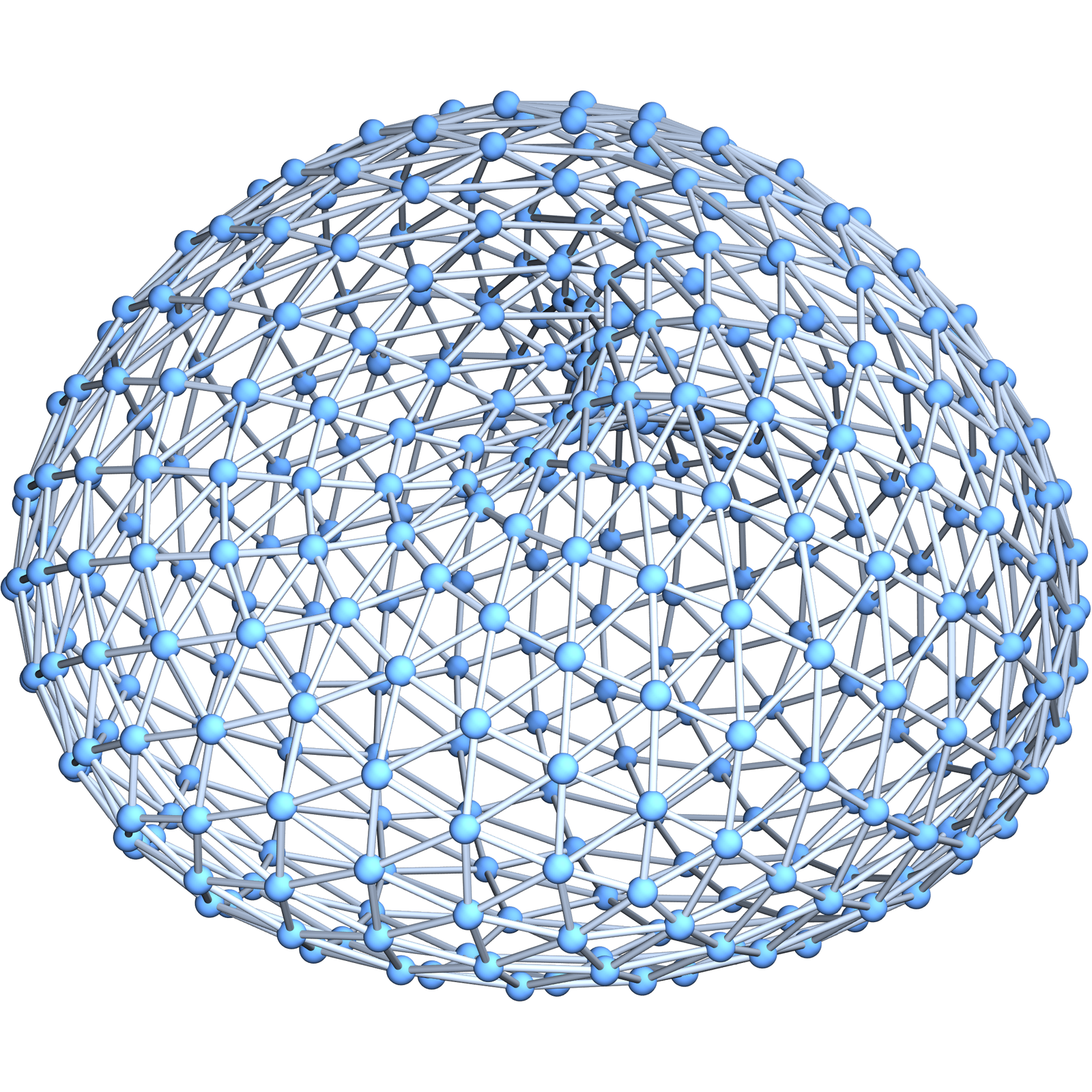}}
\label{Barycentric}
\caption{
We see to the left the 4'th soft Barycentric refinement of a wheel graph 
with chromatic number $c(G)=c(\phi^4(G))=4$ 
and the 3'rd softly refined projective plane $G$ 
with $c(\phi^3(G))=5$. No 2-manifold with $c(\phi(G))<c(G)$ is known.
Giving such an example would settle a conjecture of Albertson and Stromquist
\cite{AlbertsonStromquist} which says $c(G) \leq 5$ for 2-manifolds. 
The smallest manifold with $c(\phi(G))<c(G)$ we know of is the 
$q=5$-sphere $G=C_5 \oplus C_5 \oplus C_5$ (the graph join of three $1$-spheres) 
with $c(G)=3 c(C_5)=3*3=9$ for which our result shows $c(\phi(G))=q+2=7$.
}
\end{figure} 

\paragraph{}
Let us first summarize what we will prove about the {\bf chromatology} of softly 
refined manifolds $G$. The {\bf chromatic number} of $G$ is defined as the
chromatic number of the 1-skeleton complex of $G$. It is the minimal number
$c \geq 1$ for which there is a locally injective function $g:V(G) \to \{1,2,\dots, c\}$,
from the vertex list $V(G)$ of $G$ ($V(G)$ is the set of elements in $G$ with 
cardinality $1$)) to a finite set of colors,
where {\bf locally injective} means that $g(a) \neq g(b)$, 
whenever $(a,b)$ is in the edge set $E$, the sets of cardinality $2$ in $G$.
If $G$ has maximal dimension $q$, then $c \geq q+1$.
But it can be much larger in general. There are complexes with $q=1$ 
(equivalently triangle-free graphs),
for which the chromatic number can be arbitrarily large. 
We always have $c(G_1) = q+1$ for any Barycentric refined graph $G_1$ because 
$g(x)={\rm dim}(x)+1 \in \{1,2 \dots, q_1\}$ 
is an explicit coloring of the vertices $V_1$ of 
$G_1$ ($V_1$ are the simplices of $G$): given two simplices $a,b \in V_1$, 
the relation $a \subset b, a \neq b$ forces $g(a)<g(b)$. 
We will see that in the case of a soft refinement, we always have a dichotomy 
$c(G')=q+1$ or $c(G')=q+2$. 

\paragraph{}
Now to {\bf universality}: the spectrum of the Kirchhoff Laplacians $K(G_n)$ 
converges in law to a universal measure $d\mu_q$, which only depends on the 
maximal dimension $q$ of $G$ \cite{KnillBarycentric,KnillBarycentric2}.
In the case $q=1$, both in the Kirchoff as well as in the Hodge Laplacian case,
the limiting density of states is the {\bf arc-sin} distribution 
on $[0,4]$, which is the unique potential theoretic equilibrium measure on that set. 
In the case $d=2$, the limiting central limit measure already showed
Cantor-like feature. We so far only know that this measure exists for every $q$ and 
that it has finite support if and only if $q=1$.
We will note here that also soft Barycentric refinement limits exist. 
The motivation to look at limiting measures
is motivated also when comparing the relation between the number of trees and forests
in a graph, as we will see in the next paragraph. 

\begin{figure}[!htpb]
\scalebox{0.15}{\includegraphics{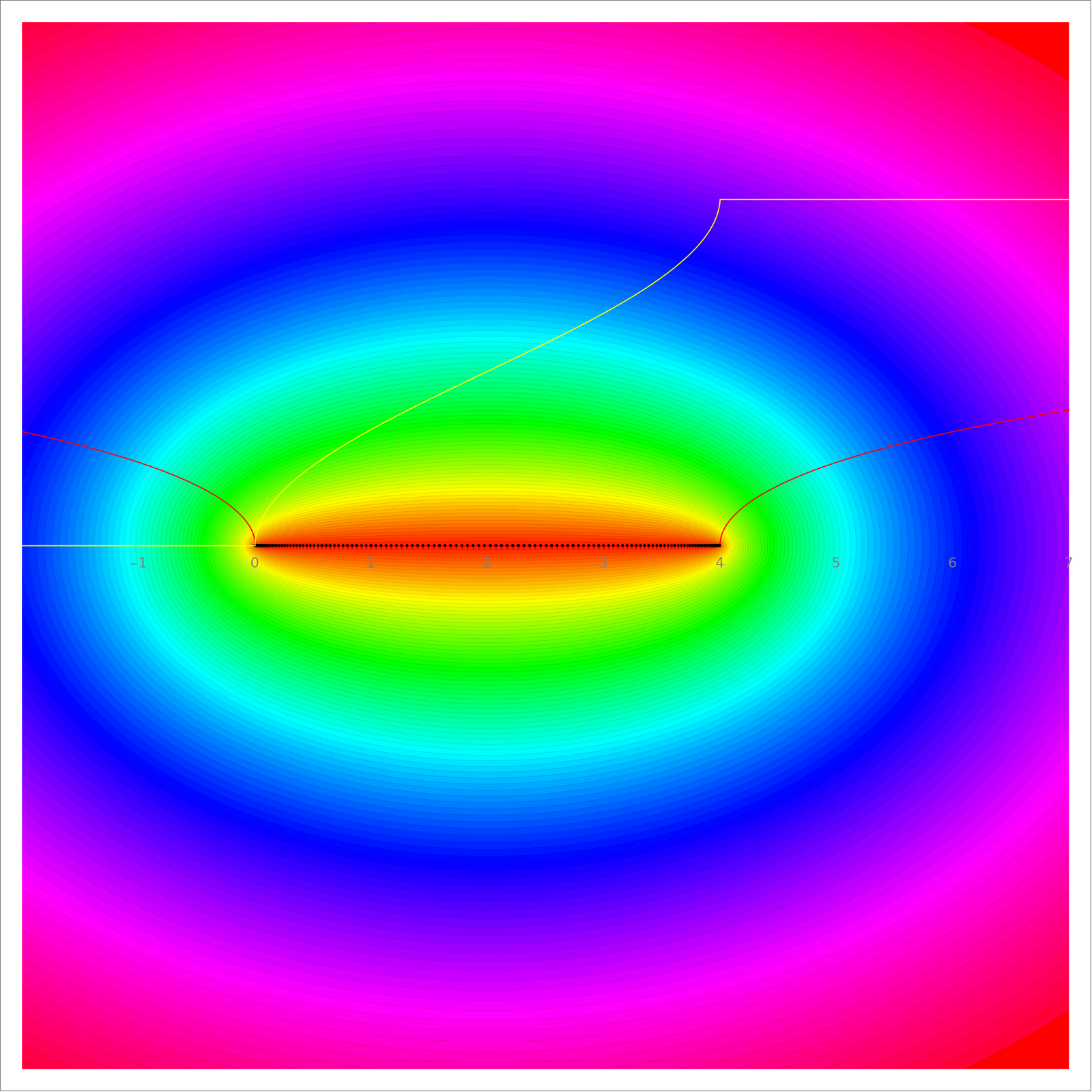}}
\scalebox{0.15}{\includegraphics{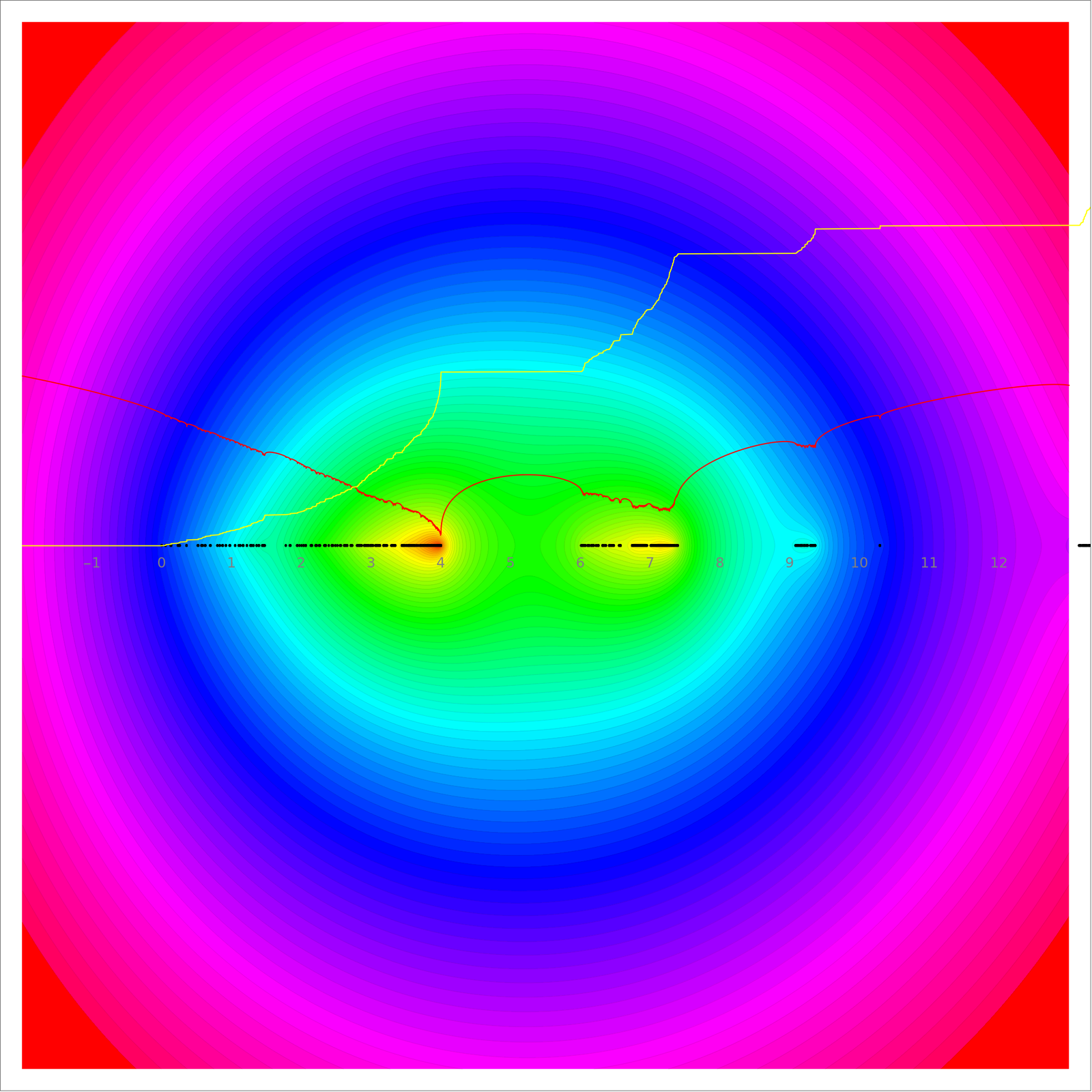}}
\label{Barycentric}
\caption{
We see approximations of the limiting measure for Barycentric 
refinement in dimension $q=1$ and $q=2$. These measures are supported on 
the real half line $[0,\infty) \subset \mathbb{R} \subset \mathbb{C}$. 
The existence of the limiting measures in any dimension $q$ is 
an older story \cite{KnillBarycentric,KnillBarycentric2}.
}
\end{figure}

\begin{figure}[!htpb]
\scalebox{0.15}{\includegraphics{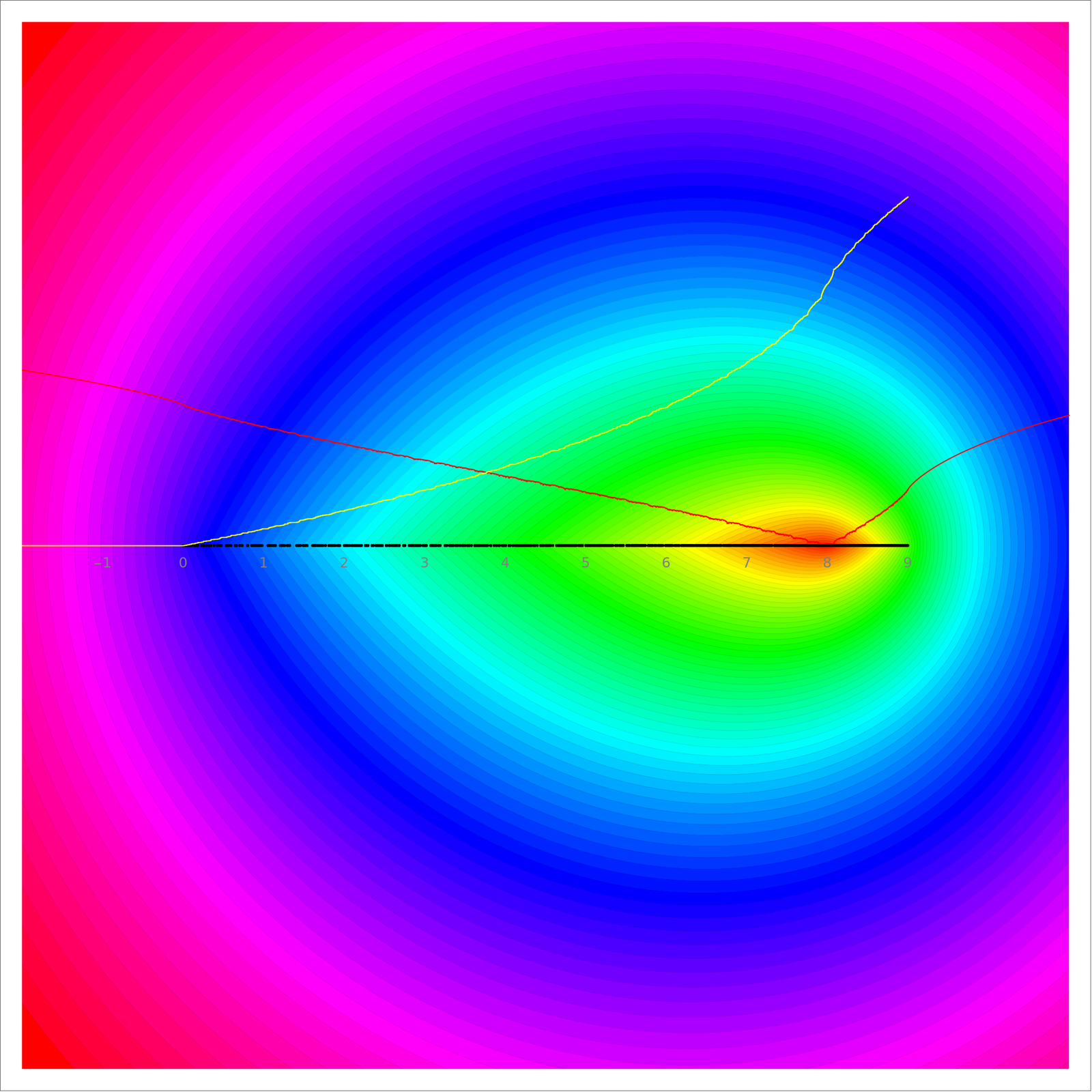}}
\scalebox{0.15}{\includegraphics{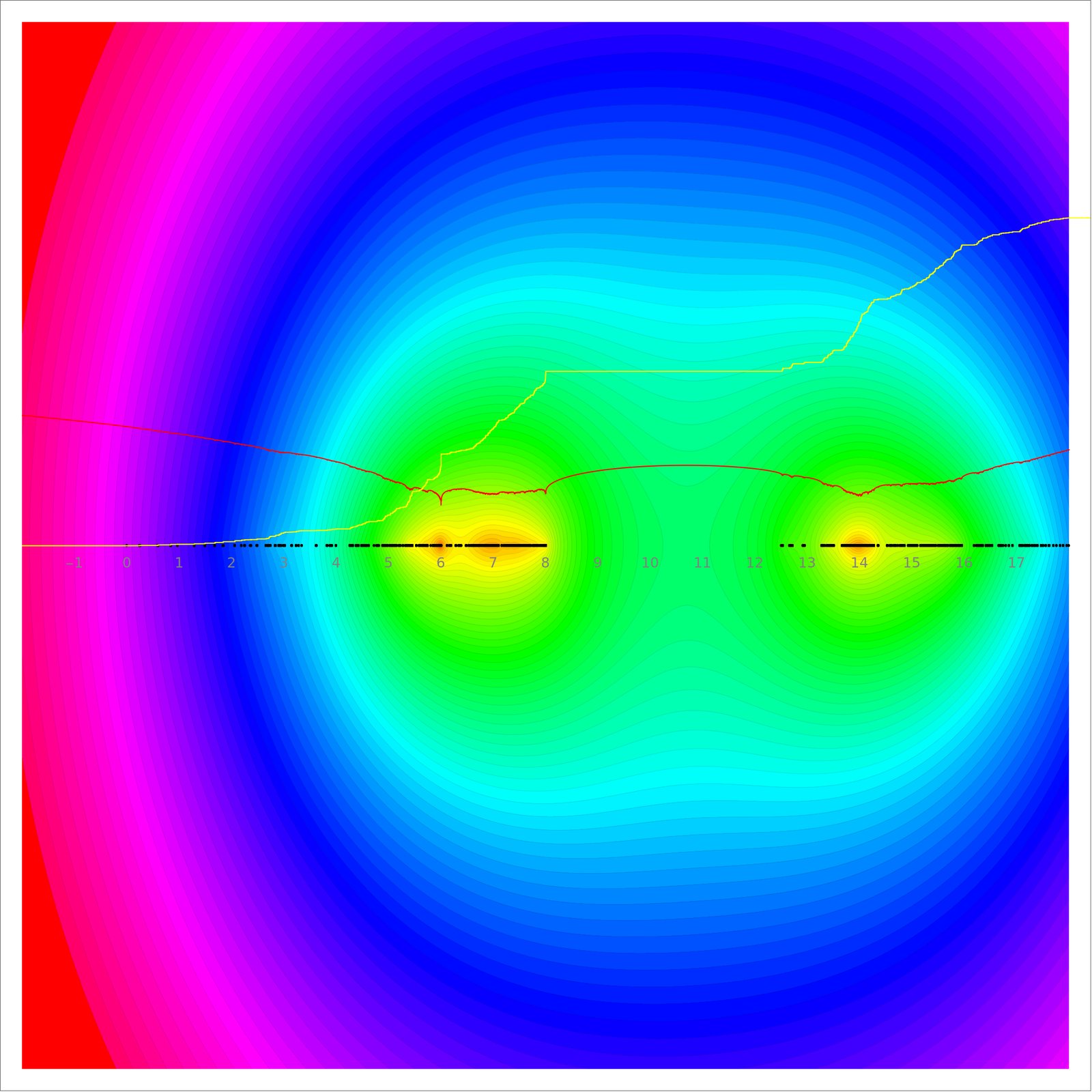}}
\label{Barycentric}
\caption{
We see approximations of the limiting measure for the soft 
Barycentric refinement limit in dimension $q=2$ and $q=3$.
This is new. For $q=2$, we have still an 
absolutely continuous measure of compact support. A Fourier transform allows to 
express the integrated density of states. It shows a 
{\bf van Hove singularity} \cite{VanHove1953}, 
which is related to critical points
of the Fourier transform $\hat{L}(x,y)=6-2\cos(x)-2\cos(y)-2\cos(x+y)$
on $L^2(\mathbb{T}^2)$, the Laplacian $L$ of the hex lattice. 
}
\end{figure}

\paragraph{}
If $\mu_q$ is either the soft or strong Barycentric limit measure in dimension $q$, 
then the {\bf potential} $f(z) = -\int_{\mathbb{C}} \log|z-z'| \; d\mu_q(z')$ 
is a complex-valued function that is analytic outside the 
support of the Barycentric limit measure $\mu_q$ in dimension $q$. 
The potential is of great interest in the Kirchhoff Laplacian case, as it encodes
exponential growth rates of trees and forests.

\paragraph{}
In the case of the Kirchhoff Laplacian, by the {\bf matrix tree theorem}, 
the {\bf tree index} $f(0)$ 
is the exponential growth rate of the number of rooted spanning trees in $G_n$ as 
$n \to \infty$. By the {\bf matrix forest theorem}, the forest index $f(-1)$ measures 
the exponential growth rate of the number of rooted spanning forests in 
$G_n$. If $\lambda$ are the eigenvalues of the Kirchhoff Laplacian $K$, then 
the {\bf pseudo determinant} ${\rm Det}(L)=\prod_{\lambda \neq 0} \lambda$ is the number of rooted 
spanning trees and ${\rm det}(1+K) = \prod_{\lambda} (1+\lambda)$ is the number of rooted 
spanning forests. The tree-forest index $\tau(G)={\rm det}(1+K)/{\rm det}(K)$
converges in the limit to a number that only depend on the size of the maximal clique in the graph.

\paragraph{}
The fact that the potential values $f(0)$ and $f(1)$ exist, follows from general 
spectral estimates $\lambda_k \leq 2d_k$, where 
$\lambda_k \leq \lambda_{k+1}$ are the eigenvalues and $d_k \leq d_{k+1}$ 
are the vertex degrees of the graph (see \cite{TreeForest}. The quest to estimate 
the potential $f(z)$ prompted the research for \cite{Knill2024}.

\paragraph{}
The {\bf facets} of $G$ are the $q$-simplices of $G$, where $q$ is the 
{\bf maximal dimension} of $G$. The {\bf boundary faces} of $G$ are the 
$(q-1)$ simplices of $G$ that are contained in precisely one facet of $G$
The other $(q-1)$-simplices are either {\bf interior faces}, the intersection of two
facets, or {\bf singular faces}, the intersection of three or more facets. 

\paragraph{}
{\bf Definition.} The {\bf soft Whitney complex} of a graph $G$ 
is the set of sub-simplices which consist of either 
$k$-simplices in $G$ for $k \neq q-1$, or then of {\bf boundary facets}, meaning 
$(q-1)$ simplices at the boundary. A boundary facet is a 
$(q-1)$ simplex which is part of exactly one $q$-simplex.
We can get to the Barycentric refinement from the soft Barycentric refinement
by edge refining all the edges connecting facets as such a subdivision.

\paragraph{}
The {\bf soft Barycentric refinement} $G_1=\phi(G)$ of a graph $G$ has 
as vertex set the elements in the {\bf soft Whitney complex} and 
connects two such points if one is contained in the other or if they 
intersect in an {\bf interior face}, a $(q-1)$ simplex which is contained 
in exactly two $q$-simplices. 

\paragraph{}
For $G=K_{q+1}$ with $q>0$ for example, the refinements $\phi^n(G)$ 
for $n>0$ are balls, meaning $q$-manifolds with $(q-1)$-spheres as boundary. 
For a manifold $G$, the growth rate of the $f$-vectors 
of the soft Barycentric refinements $G_n'$ is slower than the growth rate of the 
Barycentric refinements $G_n$. In dimension $2$, 
a second refinement $\phi^2$ agrees with the {\bf Loop refinement} 
\cite{Loop1978} defined by Loop in 1978 and the vertex degree stays bounded. 
For example, if $G$ is an icosahedron, then $G'=\phi(G)$ is 
a stellated dodecahedron. The vertex degree of $2$-dimensional manifolds does not grow.  
We have made the definitions in such a way that singular faces, 
$(q-1)$-dimensional faces which are the intersection of 3 or more
facets, do not enter in the refinement. 

\begin{figure}[!htpb]
\scalebox{0.75}{\includegraphics{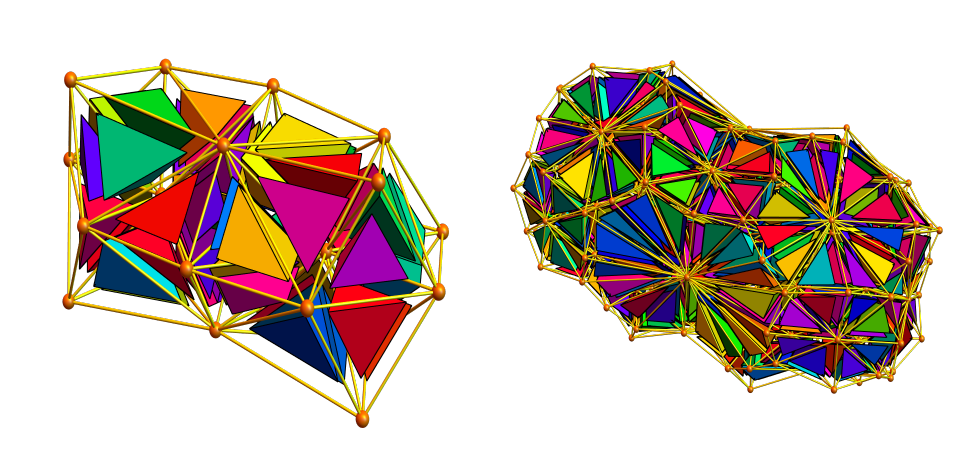}}
\label{manifold with boundary}
\caption{
The figure shows a 3-manifold with boundary and its first soft 
Barycentric refinement. Our definition of soft Barycentric 
refinement features that the boundary undergoes the 
usual Barycentric refinement. 
}
\end{figure} 

\paragraph{}
A finite simple graph is called a {\bf $q$-manifold}, 
if for every vertex $v \in V$, the unit sphere 
$S(v)$, the subgraph generated by all neighbors of $v$, is a 
$(q-1)$-sphere.  A $q$-manifold $G$ is called a {\bf $q$-sphere}, 
if there is a vertex $v$ such that $G \setminus v$ is contractible. 
A graph $G$ is called {\bf contractible}, if there
exists a vertex $v$ such that both $S(v)$ and $G \setminus v$ are both contractible.
These inductive definitions start by declaring $1=K_1$ to be contractible and 
that the empty graph $0$ to be a $(-1)$-sphere. 
A punctured $q$-sphere $G \setminus v$ if $G$ is a $d$-sphere is called a 
{\bf $q$-ball}. A {\bf $q$-manifold with boundary} is a graph such that all
unit spheres are either $(q-1)$-spheres or $(q-1)$-balls. Figure~1 illustrates the 
following lemma in dimension 2. A wheel graph is a 2-ball and 
a projective plane is a 2-manifold, meaning a 2-manifold without boundary:

\begin{lemma}
Soft Barycentric refinement preserves manifolds with and without boundary.
\end{lemma}

\begin{proof}
(i) Let us look first at manifolds without boundary. \\
A $0$-dimensional graph is always a $0$-dimensional manifold and a fixed point of $\phi$. 
In one dimensions, a connected manifold is $C_n$ with $n \geq 4$. 
It is also a fixed point of $\phi$. 
Let now $G$ be a $d$-manifold for $d \geq 2$. If in the graph $G_1$ every edge $e$ which 
came from connecting two $d$-facets is {\bf edge refined}, (meaning that a new vertex
is introduced for $e=(a,b)$ and connected to all $S(a) \cap S(b)$), 
we get the Barycentric refinement of $G$. 
In other words, the weakly refined graph $\phi(G)$ can be obtained from the 
(strong) Barycentric refinement $\psi(G)$ by edge collapses of $(q-1)$ and $q$ simplices. 
(ii) For a manifold $G$ with boundary $\delta G$, we again can just see 
$\phi(G)$ can be obtained from $\psi(G)$ by edge collapses of interior edges 
connecting a $q-1$ simplex with a $q$ simplex. 
The proof in (i) shows that the unit spheres in the interior are $(q-1)$ spheres. 
From $\delta \phi(G) = \psi \delta(G)$ and the fact that the
Barycentric refinement operation $\psi$ preserves the class of  manifolds
the boundary is a $(q-1)$- manifold.
\end{proof} 

\paragraph{}
The eigenvalues $\lambda_j$ (ordered in increasing order $\lambda_j \leq \lambda_{j+1}$ )
of the Kirchhoff Laplacian give rise to 
a {\bf spectral function} $F_G(x)=\lambda_{[n x]}$, which is encoding the eigenvalues
$0=\lambda_0 \leq \dots \leq \lambda_n$. The function $F$ is piecewise constant, 
monotone and $F_G(0)=\lambda_0$ and $F_G(1)=\lambda_n$.
The {\bf integrated density of states} $F^{-1}$ is a monotone $[0,1]$-valued 
function on $[0,\infty)$. Its derivative is in general only defined in a 
distributional sense and defines the {\bf density of states}, a probability 
measure $dk=(F^{-1})'$ on $\mathbb{R}^+$. For finite $G$, the measure $dk$ is a 
discrete pure point measure with support on the
spectrum of $L$ (which is a finite set of points). 
Point-wise convergence of $F$ implies point-wise convergence of 
$F^{-1}$ and so weak-* convergence of the density of states. 

\paragraph{}
For two graphs on the same vertex set, the {\bf graph distance}
$d(G,H)$ is the minimal number of edges which need to be modified 
in order to get from $G$ to $H$. In general no host graph is a priori given, 
in which $G,H$ can be compared. In that case, we define $n(G,H)$ to be the minimum 
number of vertices which a host graph containing both $G,H$ can have. 
This means that the distance function $d(G,H)$ can be extended to graphs which do not 
have the same vertex set; we just define $d(G,H)$ as the minimal possible distance
we can get by embedding both in a common graph with $n(G,H)$ elements.
It is possible to formulate this differently: the graph distance is the minimal 
cardinality of the symmetric difference $E(G) \Delta E(H)$ if $G,H$ are both seen 
as a sub-graph in a complete graph $K_m$. This distance does not depend 
on the size $m$ of the host graph $K_m$ as long as both $G,H$ are subgraphs of $K_m$.

\paragraph{}
The proof of the following theorem parallels the proof in the case of the 
Barycentric refinement $\psi$. The proof itself can not be carried over 
verbatim because the number of simplices of dimension $q$ and $q-1$ essentially grow 
with the same rate, rendering the multi-scale argument invalid. The argument can be 
modified easily however. We can split the graph into different identical
pieces, evolve each piece separately. The boundary part is renormalized using 
Barycentric refinement but it grows slower than the interior. 
If we start with a manifold with boundary, then the boundary can be split. 

\begin{thm} 
In any dimension $q>1$, there is a limiting law for the soft Barycentric refinement.
It is universal in the sense, that it only depends on the maximal dimension $q$ of $G$.
\end{thm}

\begin{proof}
Some of the arguments go over directly from the proof in the Barycentric case. 
First of all: (i) $||F_G-F_H||_1 \leq 4 d(G,H)/n(G,H)$.
Proof: The Kirchhoff Laplacians $L,K$ of $G,H$ satisfy
$\sum_{i,j} |L_{ij}-K_{ij}| \leq 4 d(G,H)$
because each edge $(i,j)$ affects only the four matrix
entries $L_{ij}, L_{ji},L_{ii},L_{jj}$. By the {\bf Lidskii-Last inequality}:
\cite{SimonTrace,Last1995} for any two symmetric $n \times n$ matrices $A,B$ with
eigenvalues $\alpha_1 \leq \alpha_2 \leq \dots \leq \alpha_n$ and
$\beta_1 \leq \beta_2 \leq \dots \leq \beta_n$.
For two subgraphs $G,H$ of a common graph with $n$ vertices, and Laplacians $L,H$,
the inequality gives
$||\lambda-\mu||_1 \leq 4 d(G,H)$ so that $||F_{G'}-F_{H'}|| \leq 4 d(G,H)/n$
if $G',H'$ are the graphs with edge set of $G$ and vertex set of the
host graph having $n(G,H)$ vertices. The lemma follows.

(ii) We now check the result for a single $q$-simplex $G=K_{q+1}$. 
The soft refined graph $G_n = \phi^n(G)$ is for $n>1$
a union of $q$ manifolds $G_{n,k}$ with boundary $T_{n,k} \cup S_{n,k}$, where 
$S_{n,k}$ is the n'th Barycentric refinement of the k'th $(q-1)$ boundary simplex. 
While the growth rate of the simplices in $T_n \cup S_n$ has a similar exponential growth
We use that the intersection of two different 
$G_{n,k}$ and $G_{n,l}$ is either empty or then one of the $T_{n,k}$, 
which grows with a different scale. 
If $G$ is the disjoint copy of $(q+1)$ copies of $H$, then $F(G)=F(H)$. 
There exists a constant $C_q$ such that 
$||F_{G_n}-F_{G_{n+1}}||_1 \leq C_q \frac{1}{\gamma}^n$
where $\gamma=1/(q+1)$. So, $F_{G_m}$ is a Cauchy sequence in $L^1([0,1])$
having a limit in this Banach space. The limiting $F$ produces a 
limiting measure $dk$. \\

(iii) In general, the boundary of $G_n$ grows slower than 
$G_n$, even so it is not exponentially slower. 
Let $U_n$ denote the interior, the graph generated by interior 
points.  Then $|F_{U_n} - F_{G_n}|_1 \to 0$. 
A general initial simplicial complex $G$, we write it as a union of the open set $U$
consisting of maximal simplices and the closed set $K$ which is $G \setminus U$ and
is the {\bf skeleton complex} of dimension $q-1$.   \\

If $U_n$ is the interior of the refinement $\phi^n(\overline{U})$,
$K_n = \phi^n(K)$ and $G_n=\phi^n(G)$, then $U_n \cup K_n = G_n$.
Again we have $|F_{U_n} - F_{G_n}|_1 \to 0$ and since $U$ is disjoint union of identical
parts, the result reduces to the situation in (i). 
\end{proof}

\section{Chromatic number}

\paragraph{}
We now turn to the second topic of this paper. 
The {\bf chromatic number} $c(G)$ of a graph $G=(V,E)$ is the 
minimal number of colors $C=\{1,\dots,c\}$  needed so that there exists a 
locally injective map $f: V \to C$, meaning $f(v) \neq f(w)$
if $(v,w) \in E$. The {\bf vertex arboricity} $a(G)$ of $G$ is the minimal 
number of forests that are needed to cover $V$ in such a way that every 
tree generates itself in $G$. (A subgraph $H$ of $G$ generates itself if
all pairs $(a,b) \in E(G)$ with $a,b \in V(H)$ has $(a,b) \in E(H)$.) 
The chromatic number could also be called {\bf vertex seed arboricity} 
because it asks to cover $V$ with {\bf seeds} in $G$, (a seed collection 
is forest in which every tree is a single vertex).
There is the general relation $c(G)/2 \leq a(G) \leq c(G)$:
the right inequality holds because every seed is a tree, the left
inequality holds because every tree can be colored by 2 colors, 
leading to $c(G) \leq 2 a(G)$. 
Both vertex arboricity and chromatic number are difficult quantities to compute
in general. This is in contrast to edge arboricity that can be tackled thanks to
the {\bf Nash-Williams} theorem. 

\paragraph{}
Let $\hat{G}$ denote the {\bf dual manifold} of $G$. It has the maximal $q$-simplices as 
vertices and connects two if they intersect in a $(q-1)$ simplex. 
The dual manifold is a graph
without triangles. We can attach $k$-cells to the duals of $(q-k)$-dimensional simplices and 
so get a natural dual CW complex defined by $G$. The $f$-vector of $\hat{G}$ is 
$(f_d,f_{d-1}, \dots, f_0)$. The dual graph has nice features and can be used to define 
a geodesic flow and sectional curvature as we will work out elsewhere. 
The reason to work in the dual is that for paths in $G$, the radius of 
injectivity is always $1$. For the dual graph, there is even a chance
to have a unique shortest connection between any two points. This happens for example
for the octahedron. 

\paragraph{}
Let $c(\hat{G})$ denote the chromatic number of $\hat{G}$ and let $a(\hat{G})$ 
denote the vertex arboricity of $\hat{G}$. The following result can be seen as a 
version of Gr\"otzsch's theorem, even so it does not prove for $q=2$ the
more comprehensive Gr\"otzsch's theorem for planar graphs. 

\begin{thm}[Dual 2 forest and 3 color theorem] For any $q$-manifold, 
$a(\hat{G})=2$ and $c(\hat{G}) \leq 3$. 
\end{thm} 

We will prove this by establishing a more general result. But here is a consequence
which had prompted the above theorem.

\begin{coro}
If $G$ is a $q$-manifold, then $c(\psi(G))=q+1$ and 
$c(\phi(G)) = (q-1) + c(\hat{G}) \leq q+2$.
\end{coro}
\begin{proof}
The vertices coming from simplices of dimension $q-2$ or less can be 
colored by the dimension $g(x) = {\rm dim}(x)$. 
The vertices in $\hat{G}$ that represent $q$-simplices in $G$ are then 
colored with a disjoint set of $c(\hat{G})$ colors.  
\end{proof}

\paragraph{}
This in particular shows that a {\bf 4-color theorem for soft Barycentric
refined 2-manifolds} $\phi(G)$ is no problem: first color the vertices that have
been vertices in $G$ and vertices that had been edges in $G$. 
Then use the remaining colors to color the vertices that had been triangles in 
$G$. 

\paragraph{}
A $q$-sphere $G$ is minimally $(d+1)$-colorable if and only if for all 
$(d-2)$ simplices $x$ of the dual circle $x'$ have even length. 
A $d$-manifold is $(d+1)$-colorable if and only if every closed circle has odd length. 
In the planar case, this was an observation of Heawood. In some way, it appears
in \cite{Koenig1936}. Start with coloring one maximal simplex, then the other
colors in the neighborhood are determined since $G$ is simply connected. 
(Simply connectivity is defined for q-manifolds by defining two simple paths to be homotopic if
they can be transformed into each other by {\bf simple homotopy steps}, which means taking a
triangle which contains one or two edges of the graph and replace it with the 
complement in the triangle, a closed path circling a triangle goes to a point. 
A graph is simply connected if it is homotopic to a point.)
If the graph has the property that every closed loop in the dual graph has 
even length, then no conflict appears during the coloring. 

\paragraph{}
From the chromatic bound $3$, we get two
forests covering $G$ if we can {\bf color acyclically}.
The reason is that we can take 2 colors which are not cyclic and form one forest
with them. The third color is the third forest (consisting of seeds only). 
So, we only have to show that we can color $\hat{G}$ with 3 colors in such 
a way that there are no {\bf  Kempe chains}. 
Kempe chains are cycles in the graph on which the coloring has 2-colors only. 

\paragraph{}
We prove a stronger statement by induction. 

\begin{thm}[Acyclic 3-color theorem]
For every dual $\hat{G}$ of a $q$-manifold $G$ with boundary, for which the boundary has been 
acyclically 3-colored, the acyclic 3-coloring can be extended to the interior,
preserving the acyclic property. Also any 2-forest cover on the 
boundary can be extended to a 2-forest cover of the graph in which the interior
is included. 
\end{thm}

\paragraph{}
Using induction with respect to dimension shows that
the acyclic color theorem implies the dual 3 color theorem in the case of 
a manifold $G$ without boundary. 
By taking a manifold $G$ and cutting out a ball, we get two manifolds with boundary. 
By induction with respect to dimension, the interface $(q-1)$-manifold can be
acyclically 3-colored. By the theorem, it can be extended on both sides and 
so color the manifold $G$. Now to the proof of the acyclic 3-color theorem: 

\begin{proof} 
The statement holds in dimension $q=1$ because every 1-manifold, with or without
boundary can be colored with 3 colors and in the boundary case, the color on the 
boundary (2 isolated points) can be extended to the interior. \\

Let $q$ be a minimal dimension with a counter example and let 
$n$ be the minimal number of facets of a dual q-manifold with 
boundary for which such an extension is no more possible.  \\

Take such a minimal example $G$. It must have an interior 
simplex (and so by the manifold definition has vertex degree 
$q+1$). If there was no interior simplex the color of the 
boundary $\delta G$ would already directly color $G$.
Pick an interior point $x_0 \in V(\hat{G})$. Color all cycles 
through $x_0$ with colors $A,B$, by taking the color 
$(-1)^{{\rm dist}(x,x_0)}$. For bones (minimal cycles) of odd length, this leads to double points,
requiring a third color $C$.\\

We have now colored a sub-manifold $H$ of $G$ with $3$ colors. If the entire 
manifold can not be colored, then also $G \setminus H$ can not be colored
and we  would have obtained a smaller manifold with boundary for which the boundary is 
acyclically colored by $3$ colors but for which we can not extend the coloring 
to the interior. This contradicts that $G$ was the smallest one. 
\end{proof} 

\paragraph{}
Les us add two remarks:  \\
{\bf Remark 1)} The 4-color theorem \cite{AppelHaken1,RingelMapColorTheorem} 
can not be proven as above. There is no such extension result from the boundary 
to the interior for manifolds with boundary. And that makes the 4-color theorem difficult.
The statement that any 3-coloring of the boundary of a 2-disk can 
be extended to a 4-coloring of the interior would actually be equivalent to 
the 4-color theorem. One has tried to use "minimal example" arguments 
to prove the 4-color theorem and Kempe managed to prove his {\bf 5 color theorem}
using a minimality argument. In a "cut and conquer argument" we would
like to find a path through a 2-ball which only needs 3 colors reducing 
the coloring problem to two smaller problems. But the boundary extension 
does not work. The Birkhoff diamond is a small example where the boundary is 
colored with 3 colors but where the two interior points can not be assigned 
any color. Proving the 4-color theorem in this way would require to list all
counter examples, establish that there are finitely many of them and check that
it is possible to color them nevertheless. 

\paragraph{}
{\bf Remark 2)} {\bf Gr\"otzsch's theorem} telling that every 
triangle free planar graph can be colored by 
$3$ colors is stronger than the dual 3-color theorem (when restricted to dimension $2$).
Not all planar triangle-free graphs are dual to 2-spheres. 
There are 3-connected planar, triangle free graphs which are
not the dual graphs of a 2-sphere. For example, take a cyclic $C_8$ and connect all 
even vertices to an additional point $A$ and all odd vertices to an 
additional point $B$. This graph is maximal among triangle-free planar graphs, is 
3-connected but is not the dual graph $\hat{G}$ of a 2-sphere $G$ 
because the vertex degree of the points A,B in $\hat{G}$ is 4, 
while it should be 3 for any dual of a 2-manifold.

\begin{figure}[!htpb]
\scalebox{0.25}{\includegraphics{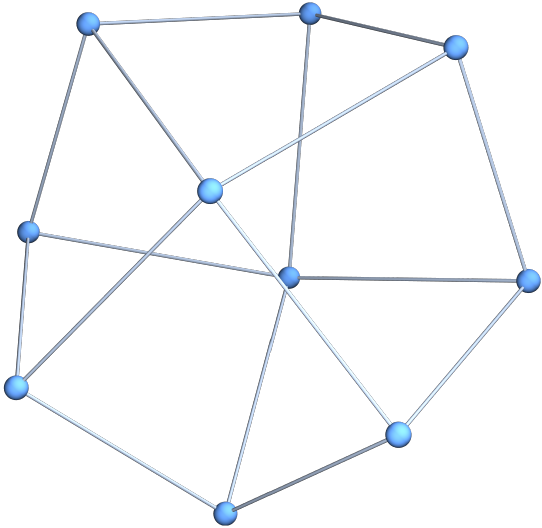}} 
\label{nonmanifold}
\caption{
A 3-connected planar triangle-free graph that is not the
dual graph of a 2-manifold. The reason is that it is not 3-regular. 
It is a subgraph of the prism graph $C_8 \oplus \overline{K_2}$.
Gr\"otsch's theorem applies to this graph, but not to the dual 3 color theorem. 
}
\end{figure}

\section{Fisk complex}

\paragraph{}
The {\bf Fisk complex} of a $q$-manifold $G$ is the $(q-2)$-dimensional 
simplicial complex generated by the collection of $(q-2)$-simplices $x$
for which the dual simplex $\hat{x} = S(x_0) \cap S(x_2) \cap \dots \cap S(x_{q-2})$ 
(which always is a cyclic graph with $4$ or more elements) has an odd number
of vertices. If $G$ is simply connected and oriented and the Fisk complex
of $G$ is empty, then we can color $G$ with the minimal $q+1$ colors. 
The Barycentric refinement destroys the entire Fisk complex as it doubles the size
of the dual spheres. For soft Barycentric refinement however, this is not the 
case and the refinement process also refines the Fisk complex. 

\paragraph{}
For a 2-manifold, the {\bf Fisk complex} is the collection of vertices on which the 
vertex degree is odd. For general $q$, it can be written as a union of $(q-2)$-manifolds. 
For an icosahedron for example, all vertices belong to 
the Fisk complex. The Fisk complex is empty if and only if the graph is Eulerian. 
The Barycentric refinement of a 2-manifold doubles the vertex degree of any 
given vertex and introduces new vertices of vertex degree 4 or 6. 
The soft Barycentric refinement does not change because 
the number of odd degree vertices stays the same. 
For a 3-manifold $G$, the Fisk complex is a one dimensional complex which 
is a union of closed curves in $G$. 

\paragraph{}
For a {\bf 3-manifold} $G$, the Fisk complex $F$ is the set of edges $x=(a,b)$ 
for which the dual sphere $x'=S(a) \cap S(b)$ has an odd number of vertices. 
The set of these simplices union of closed curves because by the Euler Handshake 
formula, the number of odd degree vertices is even in any graph. 
Also for 3-manifolds, the edge degree spectrum in the form of the 
curvature values at these edges does not grow when doing soft Barycentric 
refinements. Asymptotically, the contributions of the Fisk complex does not
matter. Since in the limit we will have the same number of degree 4 and degree 6
edges, the average edge degree converges to 5. 
For example, if $G=K_4$, then the edge degrees of $\psi(G)$ in the interior
are 4,5 or 6. For $\psi^4(G)$, there are only 576 degree 5, but 35184 degree 4 and 
34384 degree 6 edges. (There are also 5184 degree 3 edges but they are at the 
boundary of $\phi^4(G)$. 

\begin{figure}[!htpb]
\scalebox{0.75}{\includegraphics{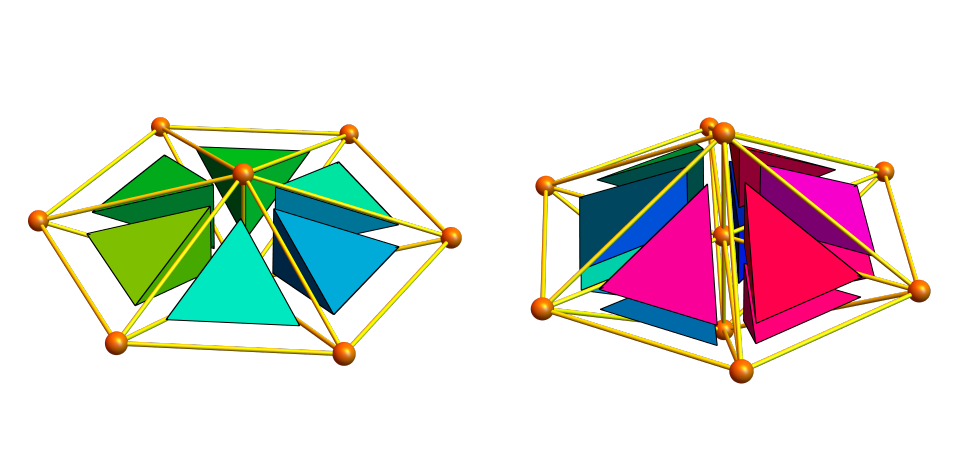}}
\label{codimension 2 simplex}
\caption{
For a 3-manifold, a codimension-2 simplex is an edge $x$.
The dual $\hat{x}$ is a cyclic polytop $C_n$, where $n$
is the number of tetrahedra in $G$ that contain $x$. After a soft
Barycentric refinement, the edge $x$ has doubled, but
each has the same dual. In the interior, the Barycentric refinement
is more tame. 
}
\end{figure} 

\paragraph{}
For a 4-manifold, the Fisk complex is the set of {\bf triangles} $x$ for which 
the dual sphere $x'=S^+(x)$ of $x$ 
(a circular graph with $4$ or more vertices) has an 
odd number of vertices. For a general $q$-manifold, the Fisk complex is generated 
by $(q-2)$-simplices for which the dual circle has odd length. 
Take a triangle $x$ in the Fisk complex. At every of the edges, there is at least one
other triangle of the Fisk complex attached. Continue building like this
a manifold. Now start with an other triangle not yet covered. There are many ways
how we can build up the Fisk complex. 
The total Euler characteristic can depend on the decomposition. 

\paragraph{}
Let us remark that the Fisk complex of a $q$-manifold is a union of $(q-2)$-manifolds
but that the decomposition into $(q-2)$ manifolds is not unique in general. 
Lets illustrate this for $q=3$: given an edge $(a,b)$ in the Fisk complex $F$, it 
defines a point $a$ in $S(b)$ which has odd degree. 
There must be an other point $c$ of odd degree so that we
can continue the path $(a,b,c)$.
For $q=4$, given a triangle $x=(a,b,c)$ in the Fisk complex $F$, 
it defines an edge $y=(a,b)$ in the 3-sphere $S(c)$. The dual circle to $y$ 
in $S(c)$ is the same and odd. There is therefore a continuation of the curve 
in $S(c)$ which produces a new triangle $z=(a,b,d)$ in the Fisk complex $F$. 
We can do this for any of the three sides. 

\paragraph{}
Given a 4 manifold and a set $T$ of triangles which have the property that at every 
boundary edge there are two or more triangles attached, Then $T$ generates a 
simplicial complex that has the property that it is the union of finitely many 
2-manifolds $M_j$. The intersection 
of any two of them is either empty or a curve. 
It follows that the Euler characteristic of $T$
is the union of Euler characteristics of $M_j$. 
Is it possible that some $M_j$ are orientable while some others are not? 

\paragraph{}
{\bf Example:} the join $G$ of the icosahedron graph and a circular graph $C_n$ is 
a 4-sphere. If $n$ is even, the Fisk complex $F$ is a union of 12 octahedra. 
The union of all edges of $F$ generates the bipartite graph $K_{2,12}$. 
If $n$ is odd, then the Fisk complex $F$ is dense in $G$ in the sense that all 
edges are covered.

\section{The two dimensional case}

\paragraph{}
A $2$-manifold is a graph such that every unit sphere is a circular graph with $4$ or more
vertices. The {\bf Barycentric refinement} of a 2-manifold $G=(V,E)$ takes $V'=V \cup E \cup F$
as vertices and takes as $E'$ the set of pairs $(x,y)$ such that $x \subset y$ or
$y \subset x$. If $f= [ |V|,|E|,|F| ]$ is the $f$-vector of $G$, then
$f(G') = A f(G) = \left[ \begin{array}{ccc} 
                  1 & 1 & 1 \\
                  0 & 2 & 6 \\
                  0 & 0 & 6 \\
                 \end{array} \right] f(G)$. Because the eigenvalues of $A^T$
is the vector $[1,-1,1]$, the Euler characteristic of $G'$ and $G$ are the same. The vertex
degrees double in each step. The linear rule $A$ holds independently of whether
$G$ is a manifold or not. 

\paragraph{}
The  {\bf soft Barycentric refinement} of a $2$-manifold takes $V'=V \cup F$ as vertices
and $E'$ as the set of pairs $(x,y)$ such that $x \subset y$ or $y \subset x$ or $x \cap y$
is in $E$.  Now the $f$-vectors transform as
$f' = \left[ \begin{array}{ccc} 
                  1 & 0 & 1 \\
                  0 & 1 & 3 \\
                  0 & 0 & 3 \\
                 \end{array} \right] f$ but this is only true if there are no
boundary faces. The Euler characteristic satisfies $\chi(G)=\chi(G')$.
Unlike for Barycentric refinement, the soft Barycentric refinement
does not have a universal matrix $A$ which works for all complexes. 

\paragraph{}
The soft Barycentric refinement preserves the class of flat tori and 
makes them larger and larger.  In general, unlike for the Barycentric 
refinement, the vertex degrees do not grow in dimension $2$.
In the limiting case, we approach the hexagonal lattice as almost all 
vertices have vertex degree 6. We can describe the limiting measure:

\begin{lemma}
For the hex region (an infinite graph), the Laplacian is 
diagonal after applying a Fourier transform.  It is equivalent 
to the multiplication operator with $f=6-2\cos(x)-2\cos(y)-2\cos(x+y)$.
\end{lemma} 

\begin{figure}[!htpb] 
\scalebox{0.75}{\includegraphics{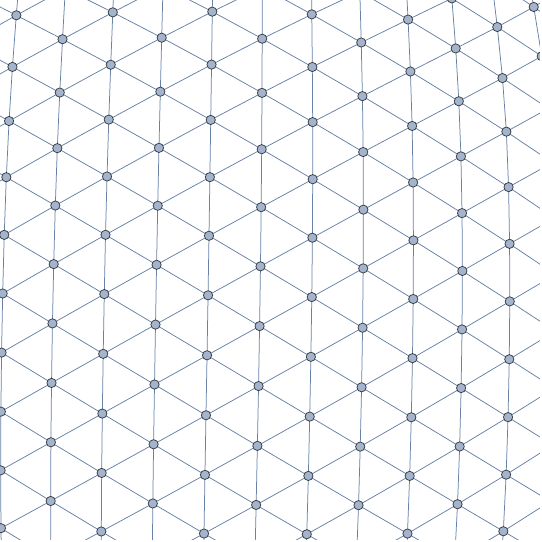}}
\scalebox{0.75}{\includegraphics{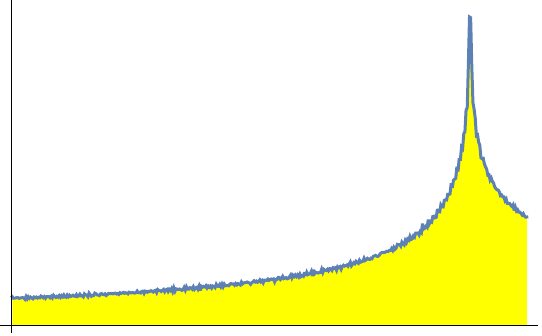}}
\label{Barycentric}
\caption{
The {\bf infinite hexagonal lattice} $\Lambda$ can be seen as the soft Barycentric limit in 
two dimensions. Its Laplacian is a bounded linear operator on $l^2(\Lambda)$ 
equivalent to the multiplication operator with 
$f=6 - 2 \cos(x) - 2 \cos(y) - 2 \cos(x + y)$ on $L^2(\mathbb{R}^2/\Lambda)$. 
The function $f$ has range $[0,9]$ and has critical points at the three points $\{0,8,9 \}$. 
The hyperbolic critical points of the function $f$ lead to the {\bf Van Hove singularity} 
in the density of states of the hexagonal lattice. Van Hove \cite{VanHove1953} already linked 
in 1953 such singularities with the Morse theory of $f$. 
}
\end{figure}

\paragraph{}
The universal limiting measure in the 2-dimensional soft Barycentric limit 
case has compact support like the 1-dimensional standard Barycentric limit.  
We see experimentally that the universal Barycentric measure $\mu_q$ in dimension $q$ 
has affinities with the spectral type of the universal soft Barycentric measure 
$\nu_{d+1}$ in dimension $q+1$. This holds in dimension $q=1$,
where Fourier theory allows to describe both measures. We can define $\mu \leq \nu$ 
on measures on the real line, if there exists $f \in L^1$ and a homeomorphism 
$g$ such that $\mu(g(x)) = f(x) \nu(x)$. Then define $\mu \sim \nu$ if $\mu \leq \nu$
and $\nu \leq \mu$. This is an equivalence relation. It honors the 
Lebesgue decomposition. If one of them has an (ac) component, the other has, if one 
of them has a pure point (pp) component, the other has. It follows from the Lebesgue decomposition 
theorem that if one of them has a (sc)-component, then the other has.

\paragraph{}
Let us look now at the chromatic number of 2-spheres.
For example, by the $4$-color theorem, any $2$-sphere $G$ has chromatic number
$3$ or $4$ and the chromatic number $3$ appears if and only if the sphere is
{\bf Eulerian}, meaning that all vertex degrees are even.
The vertex arboricity of a graph being $1$ is equivalent to the graph being a tree.
The chromatic number of a circular graph is either $2$ or $3$ and the vertex arboricity
of a circular graph is always equal to $2$.
For any 2-sphere $G$, the chromatic number of $G$ and the chromatic number of the
soft Barycentric refinement $G'$ is the same.
Proof: By the 4-color theorem, the chromatic number of the refinement $G'$
is again 3 or 4. If $G$ is Eulerian, then $G'$ is Eulerian because
the vertex degrees only can add a possible vertex degree 6. So, if
$c(G)=3$, then $G$ is Eulerian and so $G'$ is Eulerian so $c(G')=3$.

\paragraph{}
In dimension $q=2$, we can explicitly color the soft Barycentric
refinement $G'=(V',E',F')$, if a coloring of the 2-sphere $G=(V,E,F)$ is known:
we have $V'=V \cup F$ and $E'=\{ (v,f), f \in F, v \in V, v \subset f \}
\cup \{ (f,g), f,g \in F, f \cap g \in E \}$. The 4-color theorem assures that
we can color with 3 or 4 colors. 
If $c: V \to \{0,1,2,3\}$ is a coloring, then define $c'(v') = 3$ for all $v' \in V$
then start assigning $c(v')$ to one of the $v' \in V' \cap F$. 
(i) Lets first assume that the chromatic number of $G$ is $4$.
Now if and other $w' \in V' \cap F$ is adjacent, there are two possibilities:
either $v'$ and $w'$ as faces carry the same color triples in the coloring of $G$,
then define $c'(w')=c'(v')+1 \; {\rm mod} \; 3$
otherwise $c'(w') = c'(v') -1 \; {\rm mod} \; 3$. We have colored $G'$ with $4$ colors. 
(ii) If the chromatic number is $3$, the graph is Eulerian. 
Color the vertices of $G'$ with the color
$0$. If color one of the faces with $c(v')=1$.
If $v'$ and $w'$ are adjacent faces,  let $c'(w') = -c'(w') \; {\rm, mod} \; 3$.
Because the graph is Eulerian, this will color all faces without conflict. We have
colored $G'$ with $3$ colors. 

\begin{figure}[!htpb] 
\scalebox{0.25}{\includegraphics{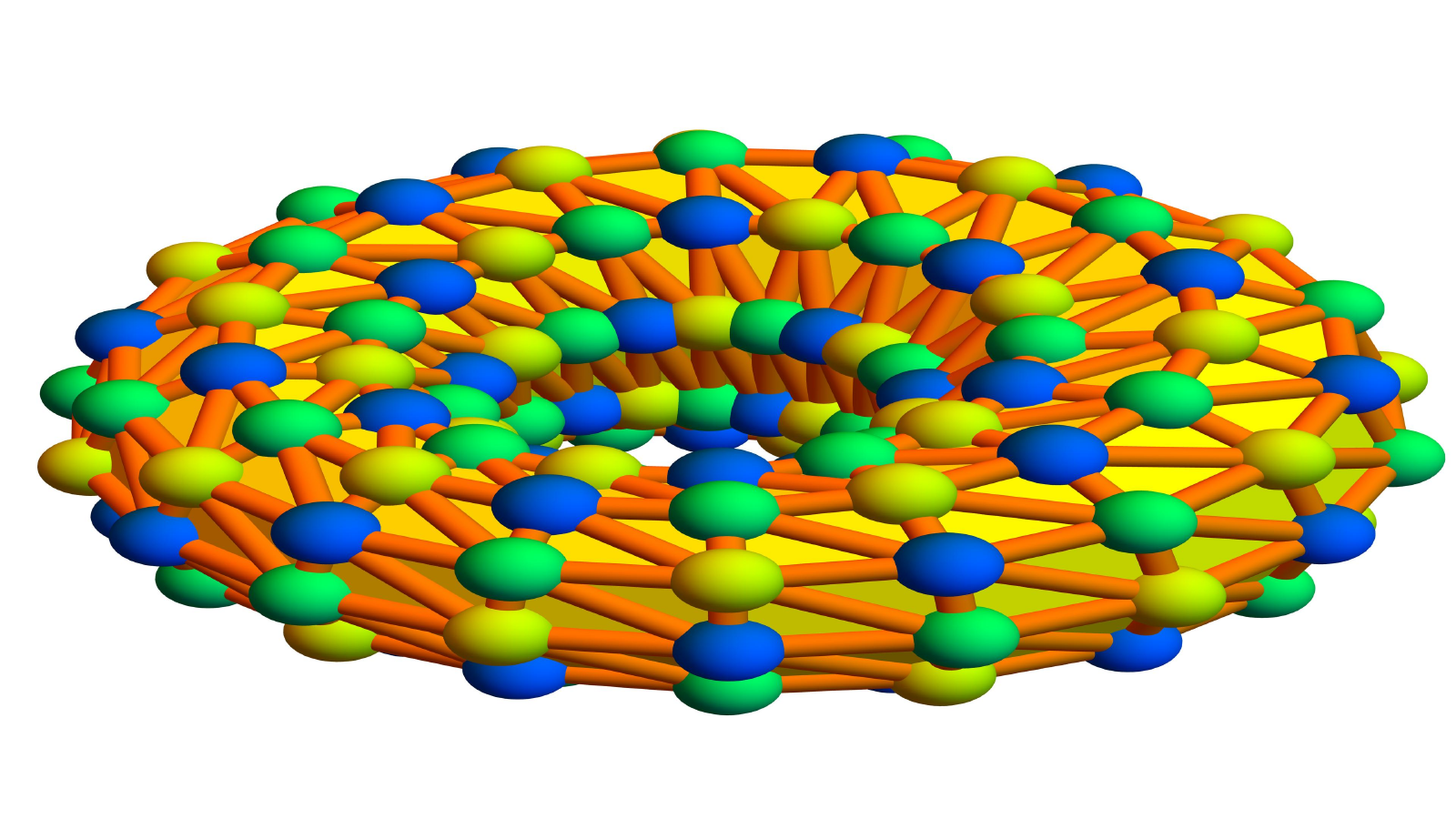}}
\label{Toruscoloring}
\caption{
A coloring of a flat Clifford torus. Clifford tori are left
invariant under soft Barycentric refinements. 
}
\end{figure}

\bibliographystyle{plain}

\begin{thebibliography}{10}

\bibitem{AlbertsonStromquist}
M.O. Albertson and W.R. Stromquist.
\newblock Locally planar toroidal graphs are {$5$}-colorable.
\newblock {\em Proc. Amer. Math. Soc.}, 84(3):449--457, 1982.

\bibitem{AppelHaken1}
K.~Appel and W.~Haken.
\newblock A proof of the four color theorem.
\newblock {\em Discrete Mathematics}, 16:179--180, 1976.

\bibitem{VanHove1953}
L.~Van Hove.
\newblock The occurrence of singularities in the elastic frequency distribution
  of a crystal.
\newblock {\em Physical Review}, 89:1189--1193, 1953.

\bibitem{KnillBarycentric}
O.~Knill.
\newblock The graph spectrum of barycentric refinements.
\newblock {{\\}http://arxiv.org/abs/1508.02027}, 2015.

\bibitem{KnillBarycentric2}
O.~Knill.
\newblock Universality for {B}arycentric subdivision.
\newblock {http://arxiv.org/abs/1509.06092}, 2015.

\bibitem{TreeForest}
O.~Knill.
\newblock The {Tree-Forest Ratio}.
\newblock https://arxiv.org/abs/2205.10999, 2022.

\bibitem{Knill2024}
O.~Knill.
\newblock Eigenvalue bounds of the {Kirchhoff Laplacian}.
\newblock {\em Linear Algebra and its Applications}, 701:1--21, 2024.

\bibitem{Koenig1936}
D.~Koenig.
\newblock {\em Theorie der endlichen und unendlichen Graphen}.
\newblock Teubner-Archiv zur Mathematik 6. Teubner, 1936.

\bibitem{Last1995}
Y.~Last.
\newblock Personal communication.
\newblock At Caltech, 1995.

\bibitem{Loop1978}
C.T. Loop.
\newblock {\em Smooth Subdivision Surfaces based on triangles}.
\newblock University of Utah, 1978.

\bibitem{RingelMapColorTheorem}
G.~Ringel.
\newblock {\em Map Color Theorem}, volume 209 of {\em Grundlehren der
  mathematischen Wissenschaften}.
\newblock Springer, 1974.

\bibitem{SimonTrace}
B.~Simon.
\newblock {\em Trace Ideals and their Applications}.
\newblock AMS, 2. edition, 2010.

\end{thebibliography}

\end{document}